\documentclass[10pt]{amsart}

\usepackage{amsmath, amsfonts, amsthm, amssymb, graphicx, fullpage, enumerate, subcaption, float}

\newtheorem{theorem}{Theorem}[section]        
\newtheorem{lemma}[theorem]{Lemma}
   
\newtheorem{conjecture}[theorem]{Conjecture}

\newtheorem*{main theorem}{Main Theorem}

\theoremstyle{remark}      
   
\theoremstyle{definition}  
\newtheorem{definition}[theorem]{Definition}   
      
\def\N{\mathbb{N}}

\def\R{\mathbb{R}}     
\def\Z{\mathbb{Z}}

\def\norm#1{\|#1\|}

\def\implies{\Longrightarrow}

\begin{document}
\title{Sets in $\R^d$ determining $k$ taxicab distances} 
\author{Vajresh Balaji, \quad Olivia Edwards, \quad Anne Marie Loftin, \quad Solomon Mcharo, \\ Lo Phillips,  \quad Alex Rice, \quad Bineyam Tsegaye}
 
\begin{abstract} We address an analog of a problem introduced by Erd\H{o}s and Fishburn, itself an inverse formulation of the famous Erd\H{o}s distance problem, in which the usual Euclidean distance is replaced with the metric induced by the $\ell^1$-norm, commonly referred to as the \textit{taxicab metric}. Specifically, we investigate the following question: given $d,k\in \N$, what is the maximum size of a subset of $\R^d$ that determines at most $k$ distinct taxicab distances, and can all such optimal arrangements be classified? We completely resolve the question in dimension $d=2$, as well as the $k=1$ case in dimension $d=3$, and we also provide a full resolution in the general case under an additional hypothesis.  \

\end{abstract}

\address{Department of Mathematics, Millsaps College, Jackson, MS 39210}
\email{balajv@millsaps.edu} 
\email{edwarof@millsaps.edu} 
\email{loftiam@millsaps.edu}
\email{mcharsk@millsaps.edu}   
 \email{philllg@millsaps.edu} 
\email{riceaj@millsaps.edu} 
\email{tsegabl@millsaps.edu} 

\maketitle 
\setlength{\parskip}{5pt}   

\section{Introduction}

In 1946, Erd\H{o}s \cite{Erdos} asked a now famous question: given $n\in \N$, what is the minimum number of distinct distances determined by $n$ points in a plane? Denoting this minimum by $f(n)$, he proved via an elementary counting argument that $f(n)=\Omega(\sqrt{n})$, and he conjectured that the correct order of growth is $n/\sqrt{\log n}$, as attained by a $\sqrt{n}\times\sqrt{n}$ integer grid. After decades of incremental progress, this conjecture was effectively resolved in a celebrated result of Guth and Katz \cite{GuthKatz}, who established that $f(n)=\Omega(n/\log n)$.

50 years after Erd\H{o}s's original paper, Erd\H{o}s and Fishburn \cite{EF} addressed the same question from the inverse perspective, and aspired to precise results in fixed cases rather than general asymptotic results. Specifically, they investigated the following: given $k\in \N$, what is the maximum number of points in a plane that determine at most $k$ distinct distances, and can such optimal arrangements be classified? This question, which we refer to as the \textit{Erd\H{o}s-Fishburn problem}, was fully resolved by Erd\H{o}s and Fishburn for $1\leq k \leq 4$, then by Shinahara \cite{Shin} for $k=5$, and Wei \cite{Wei} for $k=6$, while it remains open for $k\geq 7.$ By convention, in the quoted results and throughout this paper, $0$ is not counted as a distance determined by a set of points.

These questions can also be adapted to higher dimensions, and to alternative notions of distance. Here we focus on a particular, well-known alternative metric.
\begin{definition} For $d\in \N$ and $x=(x_1,\dots,x_d)\in \R^d$, we define the \textit{$\ell^1$-norm} of $x$ by $$\norm{x}_1=|x_1|+\cdots+|x_d|,$$ which in particular satisfies the \textit{triangle inequality} $\norm{x+y}_1\leq \norm{x}_1+\norm{y}_1$. Like every norm, the $\ell^1$-norm induces a metric on $\R^d$ by defining $\norm{x-y}_1$ to be the \textit{$\ell^1$-distance} between $x,y\in \R^d.$ 
\end{definition}
The metric induced by the $\ell^1$-norm is commonly referred to as the \textit{taxicab metric}, because it measures the length of the shortest path between two points in space, under the restriction that one can only travel in directions parallel to the coordinate axes, as if in a taxicab on a grid of city streets. For example, if two people at city intersections are separated by 3 blocks horizontally and 4 blocks vertically, then, as the crow flies, they are 5 blocks apart by the Pythagorean theorem. However, to actually make the journey without cutting through buildings, they must walk 7 blocks, which is the $\ell^1$-distance.

As noted in Chapters 0 and 1 of \cite{ErdosBook}, one can show that the minimum number of $\ell^1$-distances determined by $n$ points in $\R^d$ is $\Omega(n^{1/d})$, and this order of growth is attained by $\{1,2,3,\dots,\lceil n^{1/d}\rceil\}^d$. Therefore, in the case of the taxicab metric, the big picture asymptotic question is immediately resolved, which begs the question of whether this case can be analyzed more precisely. To begin this journey, we first consider the Erd\H{o}s-Fishburn problem in the plane with $k=1$. 

We fix two points $U,V\in \R^2$, say $U=(-1,0)$ and $V=(1,0)$. With the usual notion of distance, if any additional points can be added without determining an additional distance, those points necessarily lie on the circles of radius $2$ centered at $U$ and $V$, respectively. Those two circles intersect in only two points, and we find that they are $Q=(0,\sqrt{3})$ and $R=(0,-\sqrt{3})$. Since the distance between $Q$ and $R$ is greater than $2$, only one of the two can be added to $\{U,V\}$ while maintaining only a single distance. In summary, a set $P\subseteq R^2$ determining a single distance satisfies $|P|\leq 3$, and equality holds if and only if $P$ is the set of vertices of an equilateral triangle. 

However, even in this simplest case, the taxicab metric case diverges from that of the usual distance. With the taxicab metric, the ``circle" (which we refer to as an \textit{$\ell^1$-circle}) of radius $2$ centered at $U$ is in fact a square, rotated $45^{\circ}$ from axis-parallel, with the four sides connecting the points $(-3,0)$, $(-1,2)$, $(1,0)$, and $(-1,-2)$. Similarly, the $\ell^1$-circle of radius $2$ centered at $V$ is a square with sides connecting $(3,0)$, $(1,-2)$, $(-1,0)$, and $(1,2)$. Like the usual distance case, these two circles intersect in exactly two points, this time $Q=(0,1)$ and $R=(0,-1)$. The difference is that here $Q$ and $R$ are indeed separated by $\ell^1$-distance $2$, and hence the four-point configuration $\{U,V,Q,R\}$ determines a single $\ell^1$-distance.

\section{Main Definition and Results} \label{mainres}

Inspired by the four-point construction above, as well as additional trial and error, we define the following family of sets, which serve as our candidates for resolving the Erd\H{o}s-Fishburn problem for the taxicab metric. 

\begin{definition}For integers $d>0$ and $k\geq 0$, we define $$\Lambda_d(k)=\left\{ n=(n_1,n_2,\dots,n_d)\in \Z^d: \norm{n}_1\leq k, \ n_1+\cdots+n_d \equiv k \ (\text{mod }2)\right\}. $$ 

\end{definition}

\noindent $\Lambda_d(k)$ is the union of the integer lattice points lying on the $\ell^1$-spheres (which in dimension $d$ are $2^d$-faced polytopes) centered at the origin of \textit{every other} integer radius, starting with either $0$ or $1$ depending on the parity of $k$. The four-point configuration discussed in the introduction is $\Lambda_2(1)$, and some additional examples are pictured below.

\begin{figure}[H]
\centering
\begin{subfigure}{.5\textwidth}
  \centering
  \includegraphics[width=.5\linewidth]{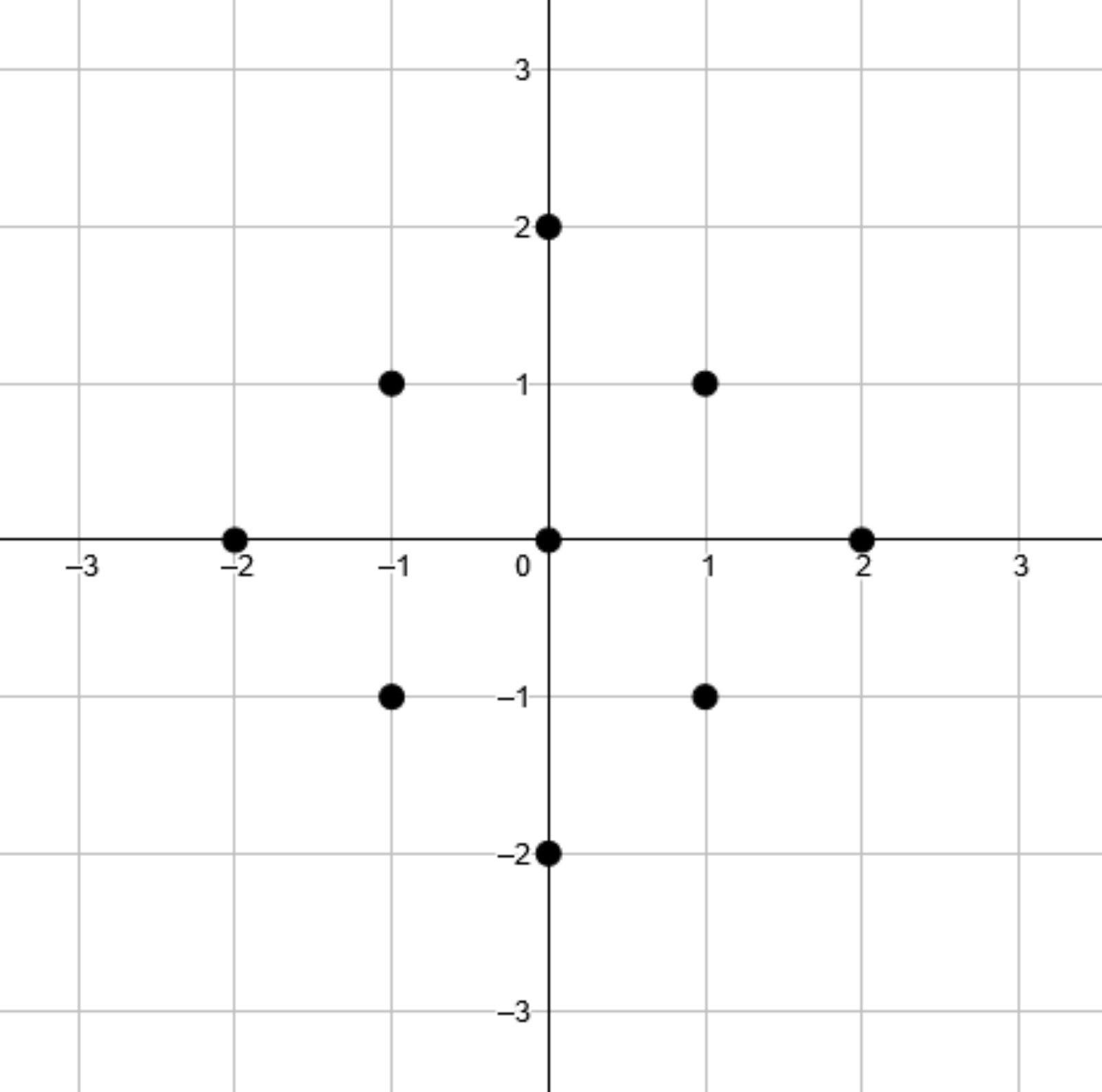}
  \caption{$\Lambda_2(2)$: $9$ points in $\R^2$ determining two $\ell^1$-distances}
  \label{fig:sub1}
\end{subfigure}%
\begin{subfigure}{.5\textwidth}
  \centering
  \includegraphics[width=.5\linewidth]{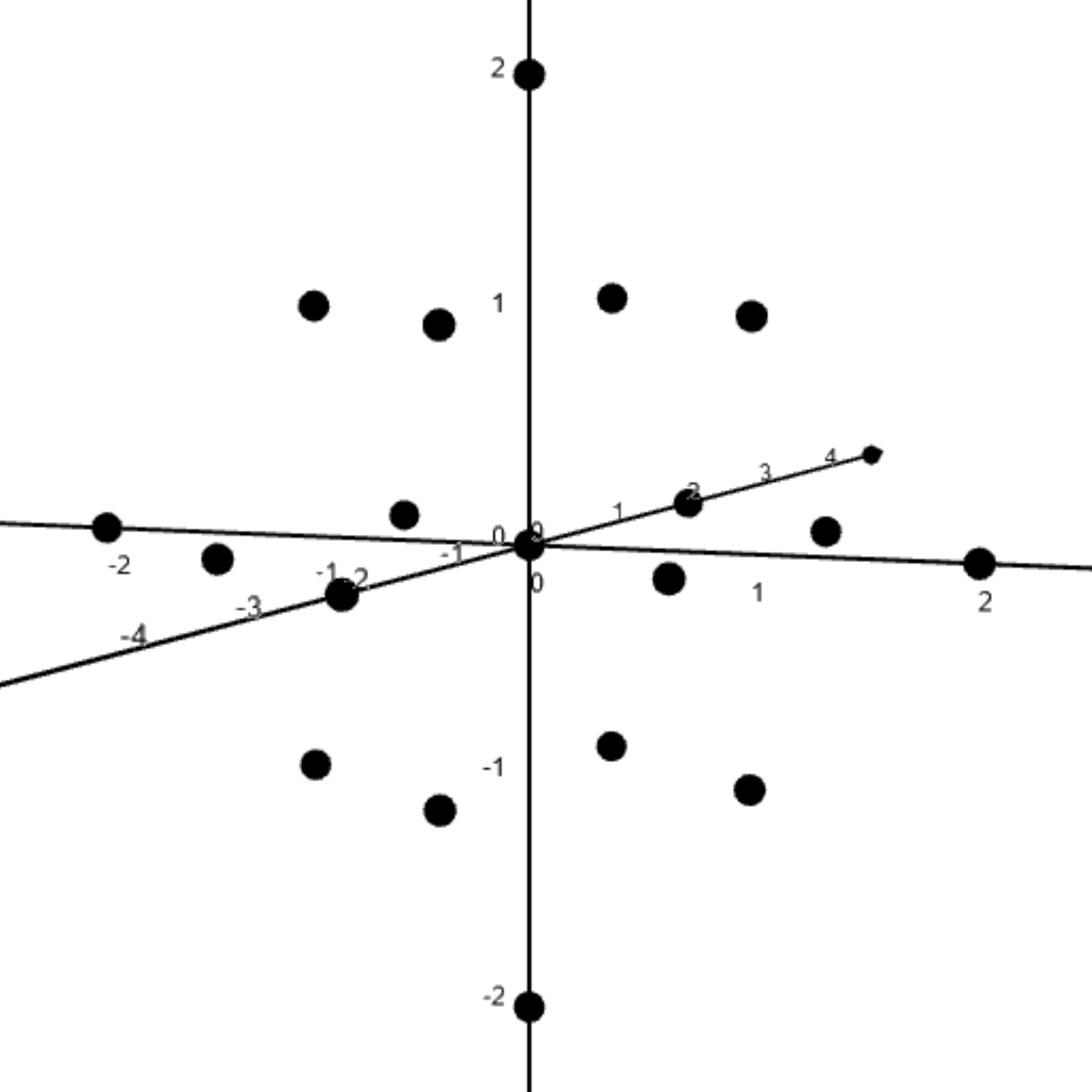}
  \caption{$\Lambda_3(2)$: $19$ points in $\R^3$ determining two $\ell^1$-distances}
  \label{fig:sub2}
\end{subfigure}
\caption{}\label{fig:test}  
\end{figure}


\noindent In Section \ref{propsec}, we establish the following properties of $\Lambda_d(k)$, including the crucial fact that it determines exactly $k$ distinct $\ell^1$-distances, the primary motivation for its definition. 

\begin{theorem} \label{prop} The following hold for all $d,k\in \N$: \begin{enumerate}[(i)] \item \label{2k} $\Lambda_d(k)$ determines exactly $k$ distinct $\ell^1$-distances, specifically $2,4,\dots, 2k$ \\  \item \label{start} $|\Lambda_1(k)|=k+1$, $|\Lambda_d(1)|=2d$ \\ \item \label{recur} $\displaystyle{|\Lambda_{d+1}(k)|=|\Lambda_{d}(k)|+2\sum_{j=0}^{k-1} |\Lambda_{d}(j)|}$ \end{enumerate}
\end{theorem}

\noindent Parts (\ref{start}) and (\ref{recur}) of Theorem \ref{prop}, combined with known formulas for sums of powers, allow one to determine explicit formulas for $|\Lambda_d(k)|$ for any fixed $d\in \N$. We include the first few examples in the following table:  

\begin{center}
\begin{table}[H] 
\caption{Explicit Formulas for $|\Lambda_d(k)|$}
\renewcommand{\arraystretch}{1.3}
\begin{tabular}{||c | c ||} 

\hline
$d$ & $|\Lambda_d(k)|$ \\
\hline\hline
$2$ & $(k+1)^2$\\
\hline 
$3$ &  $\frac{2}{3}(k+1)^3+\frac{1}{3}(k+1)$\\
\hline 
$4$ &  $\frac{1}{3}(k+1)^4+\frac{2}{3}(k+1)^2$ \\
\hline
$5$ & $\frac{2}{15}(k+1)^5+\frac{2}{3}(k+1)^3+\frac{1}{5}(k+1) $ \\
\hline 
$6$ & $ \frac{2}{45}(k+1)^6+\frac{4}{9}(k+1)^4+\frac{23}{45}(k+1)^2 $ \\
\hline
$7$ & $  \frac{4}{315}(k+1)^7+\frac{2}{9}(k+1)^5 + \frac{28}{45}(k+1)^3 + \frac{1}{7}(k+1) $\\
\hline
\end{tabular}
\label{one}
\end{table} 
\end{center}

\noindent Some of the patterns observed in Table \ref{one} can be generalized using Faulhaber's Formula for sums of powers, as seen in the following formulation, which we also prove in Section \ref{propsec}. 

\begin{theorem}\label{ffthm} For each $d\in \N$ and each integer $k\geq 0$, we have the formula \begin{equation*}|\Lambda_d(k)|=\sum_{i=0}^{\lceil d/2 \rceil-1} a_{d,i}(k+1)^{d-2i},\end{equation*} where the coefficients $a_{d,i}$ satisfy the recursive formula \begin{equation*} a_{d,i}=2\sum_{\ell=0}^i \frac{a_{d-1,\ell}}{d-2\ell} {d-2\ell \choose 2(i-\ell)}B_{2(i-\ell)}, \end{equation*} where $B_i$ is the $i$-th Bernoulli number. In particular, we have the explicit formulas $a_{d,0}=2^{d-1}/d! $ for all $d\in \N$ and $a_{d,1}=2^{d-3}/(3(d-3)!)$ for all $d\geq 3$. 
\end{theorem}

\noindent Detailed analysis of $\Lambda_d(k)$ is perhaps of independent interest, but to make headway toward our goal, we need to address the important questions: does $\Lambda_d(k)$ have maximal size amongst subsets of $\R^d$ determining at most $k$ distinct $\ell^1$-distances? If so, is $\Lambda_d(k)$ the only such optimal arrangement? In anticipation of the latter question, we observe that any optimal arrangement can undergo any scaling, or any transformation that preserves the $\ell^1$-norm, and remain optimal, leading to the following definition.


\begin{definition} \label{simdef} For $d\in \N$, we say that two subsets of $\R^d$ are \textit{$\ell^1$-similar} if one can be mapped to the other via a composition of translations, reflections about coordinate hyperplanes, dilations, and coordinate permutations, as these transformations either preserve or uniformly scale collections of $\ell^1$-distances.
\end{definition}

\noindent  We note that the list of transformations in Definition \ref{simdef} does not include rotations, because, unlike the usual Euclidean metric, the taxicab metric is \textit{not} invariant under rotation, unless the rotation can alternatively be obtained through reflection about coordinate hyperplanes and permutation of coordinates. This fact rears its head in our exploration of the taxicab metric in higher dimensions, and plays a key role in our discussions in Section \ref{hdsec}.  For now, though, the following result established in Section \ref{2dsec} completely resolves the taxicab analog of the Erd\H{o}s-Fishburn problem in the plane.

\begin{theorem} \label{opt} If $k\in \N$ and $P\subseteq \R^2$ determines at most $k$ distinct $\ell^1$-distances, then $|P|\leq (k+1)^2$. Further, $|P|=(k+1)^2$ if and only if $P$ is $\ell^1$-similar to $\Lambda_2(k)$. 

\end{theorem}

\noindent As we discuss in Section \ref{2dsec}, the $d=2$ case is simplified by the fact that, for the purposes of analyzing distance sets, the $\ell^1$-norm in $\R^2$ is effectively the same as the $\ell^{\infty}$-norm defined by $\norm{(x,y)}_{\infty}=\max\{|x|,|y|\}$. However, this equivalence does not persist in dimension $d\geq 3$, and for this reason, our proof strategy does not immediately generalize to higher dimensions. (Although, for the interested reader, the proof does generalize to show that if $P\subseteq \R^d$ determines at most $k$ distinct $\ell^{\infty}$-distances, then $|P|\leq (k+1)^d$, and equality holds if and only if $P$ is $\ell^1$-similar to $\{0,1,2,\dots,k\}^d$.)

\noindent With considerable additional effort, we successfully get our foot into the higher-dimensional door in Section \ref{3dsec}, which assures us that the unique optimality of $\Lambda_d(k)$ is not completely dependent on a connection to the $\ell^{\infty}$-norm.

\begin{theorem}\label{singdist} If $P\subseteq \R^3$ determines a single $\ell^1$-distance, then $|P|\leq 6$. Further, $|P|=6$ if and only if $P$ is $\ell^1$-similar to $\Lambda_3(1)$.
 
\end{theorem}

\noindent \textbf{Remark on previous work for $k=1$.} After the initial posting of this paper to the arxiv server, we were alerted to previous work done in the $k=1$ case (referred to as \textit{equilateral sets}) in a variety of metric spaces, including $\R^d$ with the taxicab metric (referred to as \textit{rectilinear space}). Specifically, Theorem \ref{singdist} above follows from Corollary 4.2 of \cite{Band}, due to Bandelt, Chepoi, and Laurent, while Koolen, Laurent, and Schrijver \cite{Koolen} showed that if $P\subseteq \R^4$ determines a single $\ell^1$-distance, then $|P|\leq 8=|\Lambda_4(1)|$. This partially settles a question of Kusner (see Problem 0 in \cite{Guy}), who asked if $|P|\leq 2d=|\Lambda_d(1)|$ holds for subsets of $\R^d$ determining a single $\ell^1$-distance, and this remains open for $d\geq 5$. Conjecture \ref{optcon} below can be thought of as a precise, multi-distance generalization of Kusner's question. While the conclusion of Theorem \ref{singdist} was known previously, we believe our alternative, elementary proof given in Section \ref{3dsec} remains of interest.

\noindent In Section \ref{hdsec}, we explore the question of what additional hypotheses are required to prove the optimality of $\Lambda_3(k)$ for all $k\in \N$, or even $\Lambda_d(k)$ in full generality. We find that the proof of Theorem \ref{opt} can be fully adapted with a seemingly mild additional assumption, leading us to make the following general conjecture.

\begin{conjecture} \label{optcon} If $d,k\in \N$ and $P\subseteq \R^d$ determines at most $k$ distinct $\ell^1$-distances, then $|P|\leq |\Lambda_d(k)|$. Further, $|P|=|\Lambda_d(k)|$ if and only if $P$ is $\ell^1$-similar to $\Lambda_d(k)$. 

\end{conjecture} 

\section{Properties of $\Lambda_d(k)$: Proof of Theorems \ref{prop} and \ref{ffthm}} \label{propsec}

We begin this section by proving the essential properties of $\Lambda_d(k)$ that make it a worthy candidate for resolving the Erd\H{o}s-Fishburn problem for the taxicab metric. 

\subsection{Proof of Theorem \ref{prop}} Fix $k\in \N$.  For (\ref{2k}), fix $d\in \N$, note that by definition of $\Lambda_d(k)$, we have $\norm{n}_1\leq k$ for all $n\in \Lambda_d(k)$. In particular, for any $n,m\in \Lambda_d(k)$, we have by the triangle inequality that $$\norm{n-m}_1 \leq \norm{n}_1+\norm{m}_1 \leq k+k=2k. $$ 

\noindent Further, $\norm{n-m}_1=|n_1-m_1|+\cdots+|n_d-m_d|$ is certainly an integer, and by definition of $\Lambda_d(k)$, and the fact that an integer is congruent to its absolute value modulo 2,  we have \begin{align*} |n_1-m_1|+\cdots+|n_d-m_d| &  \equiv  n_1-m_1+\cdots +n_d-m_d \\ & \equiv (n_1+\cdots+n_d)-(m_1+\cdots+m_d) \\ & \equiv k-k \\& \equiv 0 \ (\text{mod }2).   \end{align*} Therefore, the only possible nonzero values of $\norm{n-m}_1$ are $2,4,\dots, 2k$, and for each $1\leq j \leq k$ the distance $2j$ is attained between the points $(j,0,\dots,0)$ and $(-j,0,\dots,0)$ if $j\equiv k \ (\text{mod }2)$, or between $(j,1,\dots,0)$ and $(-j,1,\dots,0)$ if $j \not \equiv k \ (\text{mod }2)$. 

\noindent For (\ref{start}), we first see that $$\Lambda_1(k)=\begin{cases}\{-k,-k+2,\dots,-1,1,\dots,k-2,k\} & \text{ if }k \text{ is odd} \\ \{-k, -k+2,\dots, -2, 0, 2,\dots, k-2, k\} & \text{ if }k \text{ is even}\end{cases}.$$ In particular, $|\Lambda_1(k)|=2\lceil k/2 \rceil=k+1$ if $k$ is odd and $|\Lambda_1(k)|=2(k/2)+1=k+1$ if $k$ is even. Secondly, we see that $\Lambda_d(1)$ is precisely $\{\pm e_i: 1\leq i \leq d\}$, where $\{e_i\}$ is the standard basis for $\R^d$.

\noindent For (\ref{recur}), we see that the possible values of the final coordinate for elements of $\Lambda_{d+1}(k)$ are integers satisfying $-k\leq x_{d+1} \leq k$. Further, for a fixed value $x_{d+1}=c$, the intersection of this hyperplane with $\Lambda_{d+1}(k)$ is $$\left\{(n_1,\dots,n_{d},c)\in \Z^{d+1}: |n_1|+\cdots+|n_{d}|\leq k-|c|, n_1+\cdots+n_{d}\equiv k-c\equiv k-|c| \ (\text{mod }2) \right\}, $$ which is in natural bijection with $\Lambda_{d}(k-|c|)$. Therefore, $$|\Lambda_{d+1}(k)|=\sum_{c=-k}^k |\Lambda_{d}(k-|c|)|=|\Lambda_{d}(k)|+2\sum_{j=0}^{k-1} |\Lambda_{d}(j)|. $$\qed

\noindent We continue by establishing a detailed formula for $|\Lambda_d(k)|$, which in particular guarantees that it has the correct order of magnitude $\Omega(k^d)$.

\subsection{Proof of Theorem \ref{ffthm}} We first note that by Theorem \ref{prop}(ii), we have $|\Lambda_1(k)|=k+1$ for all $k\geq 0$. We now fix $d\geq 2$, let $h=\lceil d/2\rceil -1$, and make the inductive hypothesis that \begin{equation} \label{indhyp} |\Lambda_{d-1}(k)|=a_{d-1,0}(k+1)^{d-1}+a_{d-1,1}(k+1)^{d-3}+\cdots+a_{d-1,h}(k+1)^{d-1-2h} \end{equation} for all $k\geq 0$. Faulhaber's formula gives  \begin{equation}\label{ff} F_{p}(n)=\sum_{j=1}^n j^p=\frac{n^{p+1}}{p+1}+\frac{n^p}{2}+\frac{1}{p+1}\sum_{i=0}^{p-1}{p+1 \choose i}B_{p+1-i} n^i, \end{equation} for all $n,p\in \N$, where $B_i$ is the $i$-th Bernoulli number. By Theorem \ref{prop}(iii), we have \begin{equation} \label{recur2}|\Lambda_d(k)|=2\sum_{j=0}^{k-1}|\Lambda_{d-1}(j)|+|\Lambda_{d-1}(k)|=2\sum_{j=0}^{k}|\Lambda_{d-1}(j)|-|\Lambda_{d-1}(k)|, \end{equation} which combines with (\ref{indhyp}) to yield \begin{align*} |\Lambda_d(k)| &=2\left(a_{d-1,0}\sum_{j=0}^k (j+1)^{d-1}+\cdots+a_{d-1,h}\sum_{j=0}^k (j+1)^{d-1-2h}\right)-|\Lambda_{d-1}(k)| \\ &=2\left(a_{d-1,0}\sum_{j=1}^{k+1} j^{d-1}+\cdots+a_{d-1,h}\sum_{j=1}^{k+1} j^{d-1-2h}\right)-|\Lambda_{d-1}(k)| \\ &=2\left(a_{d-1,0}F_{d-1}(k+1)+\cdots+a_{d-1,h}F_{d-1-2h}(k+1)\right)-|\Lambda_{d-1}(k)|.\end{align*} This tells us that we can indeed write $|\Lambda_d(k)|$ as a polynomial in $k+1$, but we wish to establish the claimed explicit and recursive formulas for the coefficients, as well as the fact that every other coefficient is zero. First we consider the $(k+1)^d$ coefficient, which only arises from the term $2a_{d-1,0}F_{d-1}(k+1)$. Since the $n^{p+1}$ coefficient of $F_{p}(n)$ is $1/(p+1)$, we have  $a_{d,0}=2a_{d-1,0}/d$. Using the base case $a_{1,0}=1$, we have by induction that $a_{d,0}=2^{d-1}/d!$, as claimed. 

\noindent Next we consider the $(k+1)^{d-1}$ coefficient, which arises from two sources: the $(k+1)^{d-1}$ coefficients of $2a_{d-1,0}F_{d-1}(k+1)$ and $-|\Lambda_{d-1}(k)|$, respectively. The former is $2a_{d-1,0}(1/2)=a_{d-1,0}$, while the latter is $-a_{d-1,0}$, which means that the $(k+1)^{d-1}$ coefficient of $|\Lambda_d(k)|$ is indeed $0$. More generally, for other coefficients corresponding to terms of the form $(k+1)^{d-1-2i}$, we use the following three facts: the $(k+1)^{d-1-2i}$ coefficient on $2a_{d-1,i}F_{d-1-2i}(k+1)=a_{d-1,i}$ by the same logic as above, the $(k+1)^{d-1-2i}$ coefficient of $-|\Lambda_{d-1}(k)|$ is $-a_{d-1,i}$, and the $(k+1)^{d-1-2i}$ coefficient of $F_{d-1-2\ell}(k+1)$ is $0$ for all $\ell<i$, because $B_n=0$ for all odd $n\geq 3$. Therefore, all $(k+1)^{d-1-2i}$ coefficients of $|\Lambda_d(k)|$ are $0$.

\noindent For the the $(k+1)^{d-2}$ coefficient, we begin by noting that a direct calculation using (\ref{ff}) and (\ref{recur2}) yields $|\Lambda_3(k)|=\frac{2}{3}(k+1)^3+\frac{1}{3}(k+1)$, hence $a_{3,1}=1/3$, which serves as the base case for another induction. Fixing $d\geq 4$ and assuming the claimed formula $a_{d-1,1}=2^{d-4}/(3(d-4)!)$ holds, the $(k+1)^{d-2}$  coefficient of $|\Lambda_d(k)|$ is formed by two contributions, from $2a_{d-1,0}F_{d-1}(k+1)$ and $2a_{d-1,1}F_{d-3}(k+1)$, respectively. 

\noindent The former is given by $$2a_{d-1,0}\left(\frac{1}{d}\right){d \choose d-2}B_2=2\cdot\frac{2^{d-2}}{(d-1)!}\cdot\frac{1}{d}\cdot\frac{d(d-1)}{2}\cdot\frac{1}{6}=\frac{2^{d-3}}{3(d-2)!},$$ while the latter is given by $$2a_{d-1,1}\cdot\frac{1}{d-2}=2\cdot\frac{2^{d-4}}{3(d-4)!}\cdot\frac{1}{d-2}=\frac{2^{d-3}}{3(d-4)!(d-2)}.$$ Therefore, we have $$a_{d,1}=\frac{2^{d-3}}{3(d-2)!}+\frac{2^{d-3}}{3(d-4)!(d-2)}=\frac{2^{d-3}+(d-3)2^{d-3}}{3(d-2)!}=\frac{2^{d-3}(d-2)}{3(d-2)!}=\frac{2^{d-3}}{3(d-3)!},$$ as claimed. 

\noindent More generally, by (\ref{ff}) and (\ref{recur2}), we see that the $(k+1)^{d-2i}$ coefficient of $|\Lambda_d(k)|$ receives a contribution from $2a_{d-1,\ell}F_{d-1-2\ell}(k+1)$ for each $0\leq \ell \leq i$. Specifically, that contribution is $$2a_{d-1,\ell}\cdot \frac{1}{d-2\ell} \cdot {d-2\ell \choose d-2i} \cdot B_{2i-2\ell}=\frac{2a_{d-1,\ell}}{d-2\ell}{d-2\ell \choose 2(i-\ell)} B_{2(i-\ell)}, $$ and the recursive formula for $a_{d,i}$ follows. \qed

\section{Optimality in Two Dimensions: Proof of Theorem \ref{opt}} \label{2dsec}

In this section, we prove the unique optimality of $\Lambda_2(k)$, in that it is the unique subset of $\R^2$, up to $\ell^1$-similarity, of maximal size amongst sets determining at most $k$ distinct $\ell^1$-distances. As referenced in Section \ref{mainres}, the proof is in part enabled by an equivalence between the $\ell^1$-norm and the $\ell^{\infty}$-norm on $\R^2$. We frame our discussion entirely in the context of the $\ell^1$-norm, but the connection is implicit in our proof, particularly the following lemma. 

\begin{lemma} \label{linfty} Let $v_1=(1,1)$ and $v_2=(-1,1)$. If $x \in \R^2$ with $x=c_1v_1+c_2v_2$, then $$\norm{x}_1=2\max\{|c_1|, |c_2|\}. $$
\end{lemma}

\begin{proof} Let $v_1=(1,1)$, $v_2=(-1,1)$, fix $x\in \R^2$, and write $x$ uniquely as $x=c_1v_1+c_2v_2=(c_1-c_2,c_1+c_2)$. By potentially reflecting over the diagonal $x_1=x_2$ and/or replacing $x$ by $-x$, both of which preserve the $\ell_1$-norm, we can assume without loss of generality that $|c_1|\geq |c_2|$ and $c_1 \geq 0$. In this case, $$\norm{x}_1=|c_1-c_2|+|c_1+c_2|=c_1-c_2+c_1+c_2=2c_1=2\max\{|c_1|,|c_2|\}.$$\end{proof}

\noindent \textbf{Remark.} As an alternative approach to Section 4, one could treat the connection between the $\ell^1$ and $\ell^{\infty}$ norms on $\R^2$ in a more explicit way. Namely, Lemma \ref{linfty} can be reframed as the statement that the map $f: (\R^2,\norm{\cdot}_1)\to (\R^2,\norm{\cdot}_{\infty})$ defined by $f(x,y)=(x+y,x-y)$ is a linear isomorphism satisfying $\norm{(x,y)}_1=\norm{f(x,y)}_{\infty}$. Therefore, any results results related to the $\ell^{\infty}$-norm can be immediately transferred to the $\ell^1$-norm via this isomorphism. In particular, the proofs that follow could be rewritten in a slightly cleaner way in the $\ell^{\infty}$ context. However, in order to maintain our hands-on approach with the taxicab metric, we have chosen to leave the proofs in their $\ell^1$ form.\\

\noindent Our main strategy for proving Theorem \ref{opt} is inspired by Erd\H{o}s and Fishburn \cite{EF}. Specifically, we suppose that $P \subseteq \R^2$ determines at most $k$ distinct $\ell^1$-distances, and we seek an upper bound on the number of points we must remove from $P$ in order to eliminate the largest $\ell^1$-distance, hence reducing to the case of $k-1$ distinct $\ell^1$-distances and allowing us to invoke an inductive hypothesis. The following sequence of lemmas formalizes this strategy. Here we define an \textit{$\ell^1$-ball} in the expected way, as the region bounded by an $\ell^1$-sphere, which for $d=2$ is an $\ell^1$-circle.

\

\begin{lemma} \label{SD} Suppose $P\subseteq \R^2$ is finite. If $D$ is the largest $\ell^1$-distance determined by $P$, then $P$ is contained in a closed $\ell^1$-ball of diameter $D$. 

\end{lemma}

\begin{proof}
Suppose $P \subseteq \R^2$ is finite. Let $v_1 =(1,1)$ and $v_2 = (-1,1)$. 
Since $\{v_1,v_2\}$ forms a basis for $\R^2$, every $x \in P$ can be written uniquely as $x= c_1 v_1 + c_2 v_2$. Choose $x_1,x_2,x_3,x_4\in P$ such that $x_1$ maximizes $c_1$, $x_2$ minimizes $c_1$, $x_3$ maximizes $c_2$, and  $x_4$ minimizes $c_2$.  Call these values $c_{1,\max}$, $c_{1,\min}$, $c_{2,\max}$, and $c_{2,\min}$, respectively. These choices contain $P$ inside of a rectangle $R$, rotated $45^{\circ}$ from axis parallel, determined by the inequalities $c_{1,\min}\leq c_1 \leq c_{1, \max}$ and $c_{2,\min} \leq c_2 \leq c_{2,\max}$.

\noindent Let $w_1 = c_{1, \max} - c_{1, \min}$ and $w_2 = c_{2, \max} - c_{2, \min}$, and assume without loss of generality that $w_1 \geq w_2$. By Lemma \ref{linfty}, we have that $\norm{x_1-x_2}_1=2w_1$ and $\norm{p_1-p_2}_1 \leq 2w_1$ for all $p_1,p_2\in R$, so $D=2w_1$ is the largest $\ell_1$-distance determined by $P$. Let $c_{2,\text{new}} = c_{2, \max} - w_1\leq c_{2,\min}$, and let $B\supseteq R \supseteq P$ be defined by the inequalities $c_{1,\min}\leq c_1 \leq c_{1, \max}$ and $c_{2,\text{new}} \leq c_2 \leq c_{2,\max}$. $B$ is a square rotated $45^{\circ}$ from axis parallel, or in other words a closed $\ell^1$-ball, of diameter $D$, as required.  \end{proof} 

\

\begin{lemma} \label{remove} If $P\subseteq \R^2$ is contained in a closed $\ell^1$-ball $B$ of diameter $D$, then the $\ell^1$-distance $D$ can be eliminated from $P$ by removing the points of P contained in any two adjacent sides of the boundary of $B$.

\end{lemma}

\begin{proof} Suppose $P\subseteq \R^2$ is contained in a closed $\ell^1$-ball $B$ of diameter $D$. 

\noindent Let $a_1, a_2$ be the left and right vertices of $B$, respectively, so in particular $\norm{a_1-a_2}_1=D$. Let $U$ denote the closed (including $a_1,a_2$) upper $\ell^1$-semicircle connecting $a_1$ and $a_2$, and let $L$ denote the open (not including $a_1,a_2$) lower $\ell^1$-semicircle connecting $a_1$ and $a_2$. Since the $\ell^1$-norm is invariant under $90^{\circ}$ rotation, it suffices to establish the conclusion of the lemma for removing the points of $P$ lying in $U$. Suppose $x_1,x_2\in P\setminus U$.

\noindent \textbf{Case 1:} At least one of $x_1,x_2$ lies in $B\setminus(U\cup L)$, which is an open $\ell^1$-ball of radius $D/2$. 

\noindent Assume without loss of generality that $x_1\in B\setminus(U\cup L)$, and let $c$ be the center of $B$. Therefore, $\norm{x_1-c}_1 < D/2$ and $\norm{x_2-c}_1\leq D/2$. By the triangle inequality, $\norm{x_1-x_2}_1\leq \norm{x_1-c}_1+\norm{c-x_2}_1 < D/2+D/2=D.$ 

\noindent \textbf{Case 2:} $x_1,x_2\in L$. After possibly reflecting, assume without loss of generality that $x_1$ is to the left of $x_2$ and $\norm{x_1-a_1}_1\leq \norm{x_2-a_2}_1$, so $x_1$ is positioned at least as high as $x_2$, By replacing $x_1$ with $a_1$, we move up and to the left, so both the horizontal and vertical components of the $\ell^1$-distance to $x_2$ get larger, hence $$\norm{x_1-x_2}_1<\norm{a_1-x_2}_1=D. $$
 
\noindent In both cases, all distances amongst points in $P\setminus U$ are strictly less than $D$, and the lemma follows.   \end{proof}

\

\begin{lemma}\label{line} If $P\subseteq \R^d$ is contained in a line and determines at most $k$ distinct $\ell^1$-distances, then $|P|\leq k+1$. Further, if $|P|=k+1$, then $P$ is an arithmetic progression, meaning the $\ell^1$-distances are $\lambda,2\lambda,\dots, k\lambda$ for some $\lambda>0$.

\end{lemma}

\begin{proof} Since $\ell^1$-distance along a straight line in $\R^d$ is just a constant multiple, depending on the direction of the line, times the standard Euclidean distance, it suffices to establish the lemma with $d=1$, for which we induct on $k$. 

\noindent The base case $k=1$ is trivial, as three points $x_1<x_2<x_3$ in $\R$ automatically determine two distances $x_2-x_1<x_3-x_1$, and any two points form an arithmetic progression. 

\noindent Now, fix $k\geq 2$, and assume that if $Q\subseteq \R$ determines at most $k-1$ distances, then $|Q|\leq k$, and further, if $|Q|=k$, then $Q$ is an arithmetic progression. Now suppose $P\subseteq \R$ determines at most $k$ distances.

\noindent Let $P=\{x_1<x_2<\cdots<x_n\}$. The $n-1$ distances $x_2-x_1<x_3-x_1<\cdots<x_n-x_1$ are all distinct, hence $n-1\leq k$, or in other words $n\leq k+1$. Further, suppose $n=k+1$. By removing $x_{k+1}$, we also remove the longest distance $x_{k+1}-x_1$, so the set $Q=\{x_1,\dots, x_k\}$ determines $k-1$ distances. By our inductive hypothesis, $Q$ must be an arithmetic progression, in other words $Q=\{x_1,x_1+\lambda,x_1+2\lambda, \dots x_1+(k-1)\lambda\}$. 

\noindent If $x_{k+1}<x_1+k\lambda$, then both $x_{k+1}-x_1>(k-1)\lambda$ and $x_{k+1}-x_k<\lambda$ are new distances not determined by $Q$. If $x_{k+1}>x_1+k\lambda$, then both $x_{k+1}-x_1>k\lambda$ and $x_{k+1}-x_2>(k-1)\lambda$ are new distances not determined by $Q$. In either case, $P$ determines at least $k+1$ distinct distances, contradicting the assumption that it determines at most $k$ distances. Therefore, $x_{k+1}$ must be $x_1+k\lambda$, and the lemma follows.\end{proof}

\begin{lemma} \label{semicirc} If $S \subseteq \R^2$ is contained in the union of two adjacent sides of an $\ell^1$-circle and determines at most $k$ distinct $\ell^1$-distances, then $|S|\leq 2k+1$. Further, if $|S|=2k+1$, then the points of $S$ on each side form an arithmetic progression containing the shared vertex.
\end{lemma}

\begin{proof}Suppose $S \subseteq \R^2$ is contained in the union of two adjacent sides of an $\ell^1$-circle and determines at most $k$ distinct $\ell^1$-distances. Assume without loss of generality that the two adjacent sides are the closed upper semicircle. We know from Lemma \ref{line} that there are at most $k+1$ points on each of the two sides. 

\noindent Further, if $|S|\geq 2k+1$, then there are exactly $k+1$ points on one side, assume the left, and at least $k$ points on the right side. Let $x_1,\dots,x_{k+1}$ denote the points of $P$ on the left side, ordered left to right, and let $y$ be any point of $P$ on the right side. We note that $$\norm{x_1-x_2}_1<\norm{x_1-x_3}_1<\cdots<\norm{x_1-x_{k+1}}_1\leq \norm{x_1-y}_1, $$ and $\norm{x_1-x_{k+1}}_1= \norm{x_1-y}_1$ is only possible if $x_{k+1}$ is the vertex shared by the two sides.  In particular, if the shared vertex is not included amongst the $k+1$ points on the left side, then at least $k+1$ distinct $\ell^1$-distances occur from the leftmost point, contradicting our assumption. 

\noindent Therefore, if $|S|\geq 2k+1$, it must be the case that there are exactly $k+1$ points on both the left and right sides, including the shared vertex, meaning in fact $|S|=2k+1$. Finally, by Lemma \ref{line}, we know that the $k+1$ points on each side must form an arithmetic progression.  \end{proof}

\noindent We are now fully armed to show the unique optimality of $\Lambda_2(k)$.

\begin{proof}[Proof of Theorem \ref{opt}] We induct on $k$.  For our base case, consider $k=0$. In order for a set to determine $0$  $\ell^1$-distances (as always, not including $0$), it can contain at most $1=(0+1)^2$ point, and if it contains a point, then it is trivially a translation of $\Lambda_2(0)=\{(0,0)\}$.

\noindent Now, fix $k\in \N$, assume the conclusion of the theorem holds for $k-1$, and suppose $P\subseteq \R^2$ determines at most $k$ distinct $\ell^1$-distances. By Lemma \ref{SD}, $P$ is contained in a closed $\ell^1$-ball $B$ of diameter $D$, where $D$ is the largest $\ell^1$-distance determined by $P$. By Lemma \ref{remove}, we can remove the distance $D$ by removing the points of $P$ that lie on the closed upper $\ell^1$-semicircle $U$ on the boundary of $B$. Since $D$ has been removed as an $\ell^1$-distance, we know that $T=P\setminus U$ determines at most $k-1$ distinct $\ell^1$-distances. By our inductive hypothesis, $|T|\leq k^2$, and if $|T|=k^2$, then $T$ is $\ell^1$-similar to $\Lambda_2(k-1)$.

\noindent Further, by Lemma \ref{semicirc}, we know that $S=P\cap U$ satisfies $|S|\leq 2k+1$, and if $|S|=2k+1$, then $S$ consists of two $(k+1)$-term arithmetic progressions, one on each side of $U$, which meet at the shared vertex. Therefore, $|P|\leq |T|+|S|\leq k^2+2k+1=(k+1)^2$, and $|P|=(k+1)^2$ if and only if $T$ is $\ell^1$-similar to $\Lambda_2(k-1)$ and $S$ is a union of two arithmetic progressions meeting at the shared vertex. Finally, the only way these two sets can be combined without creating additional $\ell^1$-distances is for $S\cup T$ to be $\ell^1$-similar to $\Lambda_2(k)$. \end{proof}

\section{Single $\ell^1$-distance in three dimensions: Proof of Theorem \ref{singdist}} \label{3dsec}

Without analogs of Lemmas \ref{linfty} and \ref{SD} in dimension $d\geq 3$, our strategy for proving Theorem \ref{opt} does not naturally generalize to higher dimensions. However, in the case of $k=1$, we make the observation that if $P\subseteq \R^d$ determines a single $\ell^1$-distance, then all but the ``southernmost" point (the point minimizing the last coordinate) of $P$ lie on a single closed upper $\ell^1$-hemisphere. The following sequence of lemmas provide a detailed investigation into how $\ell^1$-distance behaves between points on a single upper $\ell^1$-hemisphere in $\R^3$, which consists of four flat faces, one for each quadrant determined by the first two coordinates, intersecting at a single northernmost point. The three lemmas correspond to the cases where the points lie on the same face, opposite faces, or neighboring faces, respectively.

\

\begin{lemma} \label{face} Suppose $V,W\in \R^3$ with $V=(x_1,y_1,z_1)$ and $W=(x_2,y_2,z_2)$. If $\norm{V}_1=\norm{W}_1$ and $x_1x_2,y_1y_2,z_1z_2 \geq 0$, then $$\norm{V-W}_1=2\max\{|x_1-x_2|,|y_1-y_2|,|z_1-z_2|\}. $$
\end{lemma}

\

\begin{proof} Suppose $V,W\in \R^3$, $V=(x_1,y_1,z_1)$, $W=(x_2,y_2,z_2)$, $\norm{V}_1=\norm{W}_1=\lambda$, and $x_1x_2,y_1y_2,z_1z_2 \geq 0$. After reflections about coordinate planes, coordinate permutations, and relabeling $V$ and $W$ (which all preserve both sides of the equation in the conclusion of the lemma), we can assume without loss of generality that all coordinates are nonnegative and $x_1-x_2\geq |y_1-y_2| \geq |z_1-z_2|$. Since $$\norm{V}_1=x_1+y_1+z_1=\norm{W}_1=x_2+y_2+z_2=\lambda,$$ we have in particular that $(x_1-x_2)+(y_1-y_2)+(z_1-z_2)=0$. Since the largest coordinate distance is in the $x$-direction, and $x_1\geq x_2$, we must have $y_1\leq y_2$ and $z_1\leq z_2$. Therefore \begin{align*}\norm{V-W}_1&=(x_1-x_2)+(y_2-y_1)+(z_2-z_1) \\ &=x_1-x_2+y_2-y_1+(\lambda-x_2-y_2)-(\lambda-x_1-y_1) \\ &=2(x_1-x_2), \end{align*} and the lemma follows.
\end{proof}

\

\begin{lemma} \label{oppface} Suppose $V,W\in \R^3$ with $V=(x_1,y_1,z_1)$ and $W=(x_2,y_2,z_2)$. If $\norm{V}_1=\norm{W}_1=\lambda$,  $x_1x_2\leq 0$, $y_1y_2\leq 0$, and $z_1,z_2\geq 0$, then $$\norm{V-W}_1=2(\lambda-\min\{z_1,z_2\}). $$
\end{lemma}

\begin{proof} Suppose $V,W\in \R^3$ with $V=(x_1,y_1,z_1)$, $W=(x_2,y_2,z_2)$, $\norm{V}_1=\norm{W}_1=\lambda$,  $x_1x_2\leq 0$, $y_1y_2\leq 0$, and $z_1,z_2\geq 0$. After reflections about coordinate planes and relabeling $V$ and $W$, we can assume without loss of generality that $x_1,y_1\geq 0$, $x_2,y_2\leq 0$, and $z_1\leq z_2$. 

\noindent Therefore, $x_1+y_1=\lambda-z_1$ while $-x_2-y_2=\lambda-z_2$, hence \begin{align*} \norm{V-W}_1&=(x_1-x_2)+(y_1-y_2)+(z_2-z_1) \\ &= \lambda-z_1+\lambda-z_2+z_2-z_1 \\ &=2(\lambda-z_1),
\end{align*} and the lemma follows. \end{proof}

\

\begin{lemma} \label{neighbor} Suppose $V,W\in \R^3$ with $V=(x_1,y_1,z_1)$, $W=(-x_2,y_2,z_2)$, $\norm{V}_1=\norm{W}_1=\lambda$, and $x_1x_2, y_1y_2, z_1z_2 \geq 0$. If $\norm{V-W}_1=\lambda$, then $|x_1|\leq \lambda/2$.  

\end{lemma}

\begin{proof} Suppose $V,W\in \R^3$ with $V=(x_1,y_1,z_1)$, $W=(-x_2,y_2,z_2)$, $\norm{V}_1=\norm{W}_1=\lambda$,  $x_1x_2, y_1y_2, z_1z_2 \geq 0$. After reflecting about coordinate planes and scaling, we can assume $x_1,x_2,y_1,y_2,z_1,z_2\geq 0$, and $\lambda=2$. If $\norm{V-W}_1=2$, then the largest possible value of $y_2+z_2$ is $y_1+z_1+2-(x_1+x_2)$. However, since $\norm{W}_1=2$, we must have $y_2+z_2=2-x_2$, hence $2-x_2\leq y_1+z_1+2-(x_1+x_2)$, which rearranges to $x_1\leq  y_1+z_2=2-x_1$, hence $x_1\leq 1$, as required. \end{proof}

\

\noindent We now establish the unique optimality of $\Lambda_3(1)$ by conducting a case analysis based on the concentration of the points of $P$, apart from the southernmost point, on the four faces of a single closed upper $\ell^1$-hemisphere.

\

\begin{proof}[Proof of Theorem \ref{singdist}] Suppose $P\subseteq \R^3$ determines a single $\ell^1$-distance $\lambda$, and choose a point $c\in P$ that minimizes the $z$-coordinate. By translating and dilating, we can assume without loss of generality that $c=(0,0,0)$ and $\lambda=2$, and hence the remaining elements of $P$ are all contained in the closed upper $\ell^1$-hemisphere $H$ of radius $2$ centered at $(0,0,0)$. We note that the southernmost point of $\Lambda_3(1)$ is $(0,0,-1)$, so our end goal in this proof is to show that $|P|<6$ unless $P$ is $\Lambda_3(1)$ shifted up by $1$.

\noindent We consider the different ways that $P$ can be concentrated on the faces of $H$. To this end, we define $H_{++}=\{(x,y,z)\in H: x,y\geq 0\}$ and $H_{+-}=\{(x,y,z)\in H: x\geq 0, y\leq 0\}$, with analogous definitions for $H_{-+}$ and $H_{--}$. We refer to the pair $H_{++}$, $H_{--}$ as \textit{opposite} faces, and likewise for 
$H_{+-}$, $H_{-+}$. The three lemmas proven at the beginning of this section allow us to make the following assertions: 

\begin{enumerate}[(i)] \item \label{oppitem} For any pair of distinct points $U=(x_1,y_1,z_1),V=(x_2,y_2,z_2)\in P\cap H$, with $U$ and $V$ lying on opposite faces, we have by Lemma \ref{oppface} that $\min\{z_1,z_2\}=1$. \\

\item \label{sameitem} For any pair of distinct points $U=(x_1,y_1,z_1),V=(x_2,y_2,z_2)\in P\cap H$, with $U$ and $V$ lying on the same face, we have by Lemma \ref{face} that $\max\{ |x_1-x_2|,|y_1-y_2|,|z_1-z_2| \}=1$. \\

\item \label{nitem} For distinct points $U=(x_1,y_1,z_1),V=(x_2,y_2,z_2)\in P\cap H$, with $U\in H_{++}$ and $V\in H_{-+}$, we have by Lemma \ref{neighbor} that $x_1\leq 1$. Similarly, by permuting coordinates, if $U\in H_{++}$ and $V\in H_{+-}$, then $y_1\leq 1$.

\end{enumerate}

\noindent If $|P|\geq 6$, then at least five points of $P$ lie on $H$, and in particular the sizes of the four intersections of $P$ with the respective faces of $H$ must add to at least five. Therefore, the only possible arrangements of $P\cap H$ include either three points on a single face, or two points on one face and a point on the opposite face. Further, the proof is greatly simplified in the case that the ``north pole" $(0,0,2)\in P$, so we divide the argument into the following three cases:

\begin{itemize}

\item{\textbf{Case 1:}} $(0,0,2)\in P$. \\

\item{\textbf{Case 2:}}  $(0,0,2)\notin P$, and $P$ contains three points $U,V,W\in H$ such that $U$ and $V$ lie on the same face, and $W$ lies on the opposite face. \\

\item{\textbf{Case 3:}} $(0,0,2)\notin P$, and there exists a face of $H$ containing at least three points of $P$.

\end{itemize}

\noindent \textbf{Proof for Case 1:} Let $V=(0,0,2)$. For $Q=(x,y,z)\in (P\cap H)\setminus \{V\}$, we have by (\ref{oppitem}) that $z=1$. In particular, the elements of $P$ other than $(0,0,0)$ and $(0,0,2)$ take the form $(x,y,1)$ with $|x|+|y|=1$, and all pairs are separated by $\ell^1$-distance $2$. By Theorem \ref{opt}, there can be at most four such elements, and the only choice of four that works is $(1,0,1)$, $(-1,0,1)$, $(0,1,1)$, and $(0,-1,1)$. The resulting arrangement is $\Lambda_3(1)$ translated up by $1$, which establishes Theorem \ref{singdist} in this case.



\noindent \textbf{Proof for Case 2:} After reflecting about coordinate planes, we can assume that $U,V\in H_{++}$ and $W\in H_{--}$. Letting $U=(x_0,y_0,z_0)$,  $V=(x_1,y_1,z_1)$, and $W=(x_2,y_2,z_2)$, we have by (\ref{oppitem}) and (\ref{sameitem}) that $$\max\{|x_0-x_1|,|y_0-y_1|,|z_0-z_1|\}=\min\{z_0,z_2\}=\min\{z_1,z_2\}=1.$$ In particular, all three $z$ coordinates are at least $1$, and since $(0,0,2)\notin P$, we have $|z_0-z_1|<1$. Therefore, we simultaneously know that $0\leq x_0,x_1,y_0,y_1 \leq 1$ and $\max\{|x_0-y_0|,|x_1-y_1|\}=1$. 

\noindent This implies that (after potentially relabeling) either $U=(1,0,1)$ and $V=(0,y,2-y)$ for some $0<y\leq 1$ or $U=(x,0,2-x)$ for some $0<x\leq 1$ and $V=(0,1,1)$. In either case, $U\in H_{++}\cap H_{+-}$, and $V\in H_{++}\cap H_{-+}$, so $P$ contains at least one element on every face of $H$. Therefore, by (\ref{oppitem}), all points of $P$ lying on $H$ have $z$-coordinate at least $1$. Further, by the same reasoning as above, there are at most two points of $P$ on each face, and the only way two points of $P$ can lie on the same face is if they lie on opposite sides of the boundary, as with $U$ and $V$. In particular, at most four points of $P$ lie on $H$, and hence $P$ contains at most five points in total.




\noindent \textbf{Proof for Case 3:} This case gets a bit stickier, because, as some trial and error reveals, there are a variety of possible arrangements of three points on a single face of $H$ that are all separated by $\ell^1$-distance $2$. 

\noindent Focusing on $H_{++}$ for the sake of exposition, we see that our desired configuration of $\{(1,0,1), (0,1,1), (0,0,2)\}$ is merely one member of a family of arrangements obtained from the following process: 

\begin{itemize} \item Choose $x_0,y_0,z_0\geq 0$ with $x_0+y_0+z_0\leq 1$, and let $\alpha=1-(x_0+y_0+z_0)$ \\ \item Starting from $(x_0,y_0,z_0)$, construct a point by adding $1$ to one coordinate and $\alpha$ to another coordinate (so the coordinates add to $2$), then produce two additional points in a similar way by rotating the original choice of coordinates. For example, the initial choice of $U=(x_0+1,y_0+\alpha,z_0)$ uniquely determines the  two additional points $V=(x_0,y_0+1,z_0+\alpha)$ and $W=(x_0+\alpha,y_0,z_0+1)$. All of these points lie on $H_{++}$, and by (\ref{sameitem}) they are all separated by $\ell^1$-distance $2$. In fact, the only other possible set of three points yielded by this process (up to labeling) is $U=(x_0+1,y_0,z_0+\alpha)$, $V=(x_0+\alpha, y_0+1, z_0)$, and $W=(x_0,y_0+\alpha, z_0+1)$. For additional clarity, a specific example is $x_0=0.1$, $y_0=0.3$, $z_0=0.4$, hence $\alpha=0.2$, which could yield the three-point arrangements $\{(1.1,0.5,0.4),(0.1,1.3,0.6), (0.3,0.3,1.4)\}$ or $\{(1.1,0.3,0.6),(0.3,1.3,0.4),(0.1,0.5,1.4)\}$.

\end{itemize}

\noindent We hope to demystify the situation by arguing that the arrangements discussed above  are in fact the only possible arrangements. To this end, after reflections, we can assume $P$ contains three points $U,V,W\in H_{++}$, and we settle Case 3 with the following steps:

\begin{itemize} \item \textbf{Step 1 :} Show that $U,V,W$ take the form discussed above. In particular, after specifying the minimum values of each coordinate and a single point, the second and third points are uniquely determined, hence there cannot be a fourth point in $P\cap H_{++}$. \\

\item \textbf{Step 2:} Show that $P$ can contain at most one point in $(H_{+-}\cup H_{-+})\setminus H_{++}$ before necessarily reducing to Case 2. This means that any hypothetical fifth point of $P\cap H$ necessarily lies on $H_{--}$, which itself reduces the argument back to Case 2.\\

\end{itemize}

\noindent \textbf{Step 1:} Let $x_0$, $y_0$, and $z_0$ be the minimum $x$, $y$, and $z$-coordinates, respectively, attained by $U$, $V$, and $W$. In what follows, we repeatedly appeal to (\ref{sameitem}), which tells us that for every pair of points in $\{U,V,W\}$, the maximum coordinate distance is exactly $1$. In particular, the maximum $x$, $y$, and $z$-coordinates attained by $U$,$V$, and $W$ are at most $x_0+1$, $y_0+1$, and $z_0+1$, respectively, and we begin by arguing that this inequality must be equality in all three coordinates.

\noindent Suppose that this inequality is strict in at least one coordinate. By permuting coordinates and relabeling points we may assume that $U=(x_0,y,z)$, and neither of $V$ and $W$ has $x$-coordinate $x_0+1$. Therefore, the maximum coordinate distance of $1$ required by (\ref{sameitem}) must occur in either the $y$ or $z$-coordinates, and since $x_0$ is the minimum $x$-coordinate,  $V$ and $W$ must both take one of the following forms:   $(x_0+\alpha,y-1, z+(1-\alpha))$ for some $0\leq\alpha< 1$, or $(x_0+\beta, y+(1-\beta),z-1)$ for some $0\leq \beta <1$. However, no combination of these choices for $V$ and $W$ have a maximum coordinate distance of $1$ from each other, so this arrangement is impossible. Therefore, all the maxima $x_0+1$, $y_0+1$, and $z_0+1$ are indeed achieved. For the remainder of the proof, we will refer to the respective coordinate values $x_0$, $y_0$, and $z_0$ as \textit{minimum coordinates}, and we will similarly refer to the respective coordinate values $x_0+1$, $y_0+1$, $z_0+1$ as \textit{maximum coordinates}. We complete step one by considering the following three subcases.

\begin{itemize} \item \textbf{Subcase A:} Two maximum coordinates appear simultaneously in a single point. \\

\noindent Since all the points have $\ell^1$-norm $2$, this subcase necessitates that $x_0=y_0=z_0=0$, and we assume without loss of generality that $U=(1,1,0)$. Since the minimum $x$ and $y$ coordinates of $0$ must be attained, $\{V,W\}$ contains points of the form $(x,0,2-x)$ and $(0,y,2-y)$, respectively, for some $0\leq x,y\leq 1$. However, since the maximum $z$-coordinate is $1$, the only admissible choices are $x=y=1$, hence the three points are $(1,1,0)$, $(0,1,1)$, and $(1,0,1)$, which take the required form with $\alpha=1$. \\

\item \textbf{Subcase B:} Two minimum coordinates appear simultaneously in a single point.\\

\noindent Assume without loss of generality that $U=(x_0,y_0,z)$. Since $x_0$ and $y_0$ are minimum coordinates, each of $V$ and $W$ must take the form $(x_0+\alpha, y_0+\beta, z-(\alpha+\beta))$ for some $\alpha,\beta\geq 0$, and by (\ref{sameitem}) we must have $\alpha+\beta=1$.   Further, since the $z$-coordinate of both $V$ and $W$ is $z-1$ (which is hence the minimum coordinate $z_0$), the maximum coordinate distance of $1$ must occur in the first two coordinates, meaning that $\{V,W\}=\{(x_0+1,y_0,z_0),(x_0,y_0+1,z_0)\}$. In particular, the arrangement takes the required form with $\alpha=0$.   \\

\item \textbf{Subcase C:} Exactly one minimum coordinate and one maximum coordinate occurs in each point.\\

\noindent  After permuting coordinates and relabeling points we assume $U=(x_0+1,y_0+\alpha,z_0)$ where $0<\alpha=1-(x_0+y_0+z_0)<1$. In order to meet the subcase conditions, have a maximum coordinate distance of $1$ from $U$, and have $\ell^1$-norm $2$, the options for $V$ and $W$ are  $(x_0, y_0+1, z_0+\alpha)$, $(x_0, y_0+\alpha, z_0+1)$, and $(x_0+\alpha, y_0, z_0+1)$. Of these three possibilities, there is only one pair that are separated by $\ell^1$-distance $2$ from each other, hence $\{V,W\}=\{(x_0, y_0+1, z_0+\alpha),(x_0+\alpha, y_0, z_0+1)\}$, as required.

\end{itemize}

\noindent \textbf{Step 2:} Suppose $P$ contains a point $Q\in H_{-+}\setminus H_{++}$ (the argument is completely analogous for $Q\in H_{+-}\setminus H_{++}$). By (\ref{nitem}), in order for $Q$ to be separated from $U=(x_0+1,y_0+\alpha, z_0)$ by $\ell^1$-distance $2$, we must have $x_0+1\leq 1$, and hence $x_0=0$. In particular, $V=(0,y_0+1, z_0+\alpha) \in P \cap (H_{-+}\cap H_{++})$, so $P$ contains at least two points on $H_{-+}$. This means that, in order to avoid reducing to Case 2, $P$ cannot contain any elements of $H_{+-}$. 

\noindent If instead $P$ contains a second point $R\in H_{-+}\setminus H_{++}$, hence a third point in $H_{-+}$, then we fall back to our previous analysis of three points on a single face, adapted by taking negatives of all $x$-coordinates. In particular, because $V$ has $x$-coordinate $0$, which minimizes the $x$-coordinate in absolute value among the points in $P\cap H_{-+}$, either $Q$ or $R$ must maximize the $x$-coordinate in absolute value and have $x$-coordinate $-1$. Assuming $Q$ has $x$-coordinate $-1$, in order for $Q$ to be separated from $U=(1,y_0+\alpha,z_0)$ by $\ell^1$-distance $2$, we must have $Q=(-1,y_0+\alpha,z_0)$. In order for $\{V,Q,R\}$ to meet the required form for three points of $P$ on $H_{-+}$ established in Step 1, we must have $R=(-\alpha,y_0,z_0+1)$. However, in this case we see that $\norm{U-R}_1=2+2\alpha=2$, hence $\alpha=0$, which contradicts the assumption that $R\notin H_{++}$.
\end{proof}


\section{Conditional Results in Higher Dimensions} \label{hdsec}

In the remainder of our discussion, we use the terms \textit{$\ell^1$-sphere} and \textit{$\ell^1$-ball} as before, defined analogously to regular spheres and balls in $\R^d$, with the usual distance replaced by $\ell^1$-distance. In an effort to establish results in higher dimensions, we make the following observations, heavily inspired by our journey thus far:

\begin{enumerate}[(a)] \item \label{capt} As noted at the beginning of Section \ref{3dsec}, our proof of Theorem \ref{opt} does not naturally generalize to higher dimensions, because in dimension $d\geq 3$, it is not necessarily the case that if the largest $\ell^1$-distance determined by a finite set $P\subseteq \R^d$ is $\lambda$, then $P$ is contained in a closed $\ell^1$-ball of diameter $\lambda$. However, the argument in Lemma \ref{remove} does generalize to all dimensions: the distance $\lambda$ can be removed from an $\ell^1$-ball of diameter $\lambda$ by removing the closed upper $\ell^1$-hemisphere. In particular, if we somehow could capture our set inside such a ball, then by mimicking the proof of Theorem \ref{opt}, the problem is reduced to determining maximal configurations of points arranged on single closed upper $\ell^1$-hemisphere, which would then facilitate an induction on the number of distinct distances. \\ 

\item \label{ballitem} Suppose $P \subseteq \R^d$ is a finite set determining at most $k$ distinct $\ell^1$-distances, with largest $\ell^1$-distance $\lambda$. By translating and scaling, we can assume that $\lambda=2k$ and the ``southernmost point" of $P$, minimizing the $x_d$ coordinate, is $-ke_d$, where $e_d$ is the $d$-th standard basis vector. An enticing observation, particularly in juxtaposition with (\ref{capt}), is the following: if $ke_d$ is also in $P$, then, since $2k$ is the largest $\ell^1$-distance, $P$ is contained in the intersection of the closed $\ell^1$-ball of radius $2k$ centered at $-ke_d$ and the closed $\ell^1$-ball of radius $2k$ centered at $ke_d$, which is conveniently the closed $\ell^1$-ball of radius $k$ centered at the origin. \\

\item \label{hyper} Inspired by the simplicity of Case 1 in the proof of Theorem \ref{singdist}, we see that if $U$ lies on an upper $\ell^1$-hemisphere $H\subseteq \R^d$, then the $\ell^1$-distance between $U$ and the ``north pole" of $H$ is determined entirely by the $x_d$-coordinate of $U$. More specifically, if $H$ is the closed upper $\ell^1$-hemisphere of radius $k$ centered at the origin and $U=(x_1,\dots,x_d)\in H$, then $$\norm{U-ke_d}_1=|x_1|+\cdots+|x_{d-1}|+k-x_d=2(k-x_d). $$ In particular, if $ke_d\in P$ and $P$ determines only $k$ distinct $\ell^1$-distances $\{\lambda_i\}_{i=1}^k$, then the points of $(P\cap H)\setminus ke_d$ are restricted to the hyperplanes $\{x_d=c_i\}$ for $1\leq i \leq k$., where $c_i=k-\lambda_i/2$. Further, the intersection of $H$ with the hyperplane $\{x_d=c_i\}$ is $$\{(x_1,\dots,x_{d-1},c_i): |x_1|+\cdots+|x_{d-1}|=k-c_i\},$$ which is a copy of the $\ell^1$-sphere of radius $k-c_i$ centered at the origin in $\R^{d-1}$. This would allow us to analyze $P\cap H$ by inducting on dimension, analogous to the invocation of Theorem \ref{opt} during Case 1 in the proof of Theorem \ref{singdist}.

\end{enumerate}

\noindent These three items combine to a clear aspirational reality: given a finite set $P\subseteq \R^d$ that determines at most $k$ distinct $\ell^1$-distances, the largest of which is $\lambda$, letting $H$ denote the closed upper $\ell^1$-hemisphere of radius $\lambda$ centered at the ``southernmost" point of $P$, we could fully adapt the proof of Theorem \ref{opt} and induct on both $d$ and $k$, if only we could assume that the ``north pole" of H is also in $P$. If we were considering the usual Euclidean distance, this would be no obstruction at all, as we could rotate our set and assume without loss of generality that the largest distance $\lambda$ occurs parallel to the $x_d$-axis. However, since $\ell^1$-distance is not invariant under rotation, we require an additional assumption to establish a conditional version of Conjecture \ref{optcon}. The following definition, conjecture, and theorem fully formalize this conditional result, after which we conclude our discussion. 

\begin{definition} Given $P\subseteq \R^d$ and an $\ell^1$-distance $\lambda>0$, we say that $\lambda$ \textit{occurs in an axis-parallel direction} if there exists $x\in P$ and $1\leq i \leq d$ such that $x+\lambda e_i\in P$, where $e_i$ is the $i$-th standard basis vector. Further, if $P$ is bounded, we say that $P$ is \textit{axis-parallel} if the largest $\ell^1$-distance determined by $P$ occurs in an axis-parallel direction.
\end{definition}

\begin{conjecture} \label{axisp} Suppose $d,k\in \N$. If $P$ is of maximal size amongst subsets of $\R^d$ determining at most $k$ distinct $\ell^1$-distances, then $P$ is axis-parallel. The same holds within the class of sets contained in an $\ell^1$-sphere in $\R^d$.
\end{conjecture} 

\begin{theorem}\label{condinc} Conjecture \ref{axisp} implies Conjecture \ref{optcon}. 

\end{theorem}

\begin{proof} We proceed via two inductions, one on the dimension $d$ and another on the number of distinct $\ell^1$-distances $k$. We streamline the argument by defining the following propositions for each $d\in \N$ and each nonnegative integer $k$:

\begin{itemize} \item Opt$(d,k)$: $\Lambda_d(k)$ is the unique set, up to $\ell^1$-similarity, of maximal size amongst subsets of $\R^d$ determining at most $k$ distinct $\ell^1$-distances. Conjecture \ref{optcon} is precisely the statement that Opt$(d,k)$ holds for all $d,k\in \N$. \\

\item S-Opt$(d,k)$: $\Lambda_d(k)\setminus \Lambda_d(k-2)$ is the unique set, up to $\ell^1$ similarity, of maximal size amongst sets contained in an $\ell^1$-sphere in $\R^d$ determining at most $k$ distinct $\ell^1$-distances. For $k=0$ or $1$, we take the convention that $\Lambda_d(-1)=\Lambda_d(-2)=\emptyset$.\\

\item H-Opt$(d,k)$: Let $H$ denote the closed upper $\ell^1$-hemisphere of radius $k$ centered at the origin in $\R^d$. If $ke_d\in E\subseteq H$ and $E$ determines at most $k$ distinct $\ell^1$-distances, the largest of which is $2k$, then $|E|\leq |\Lambda_d(k)\cap H|$, and $|E|=|\Lambda_d(k)\cap H|$ if and only if $E=\Lambda_d(k)\cap H$.\\


\end{itemize}

\noindent For the necessary base cases, we note that  S-Opt$(1,k)$ and H-Opt$(1,k)$ trivially hold for all $k\in \N$ as $\ell^1$-spheres and $\ell^1$-hemispheres in $\R$ contain just two points and one point, respectively. Also, Opt$(d,0)$ holds trivially for all $d\in \N$ because a set determining no $\ell^1$-distances contains at most a single point. Under the assumption that Conjecture \ref{axisp} holds, we verify Conjecture \ref{optcon} by establishing the following implications:

\begin{enumerate} \item \label{SH} S-Opt$(d-1,k)$ for all $k\in \N \implies$ H-Opt$(d,k)$ and S-Opt$(d,k)$ for all $k\in \N$, so S-Opt$(d,k)$ and H-Opt$(d,k)$ hold for all $d,k\in \N$. \\

\item \label{OI} Opt$(d,k-1)$ and H-Opt$(d,k) \implies$ Opt$(d,k)$, so Opt$(d,k)$ holds for all $d,k\in \N$, as required.\\
 
\end{enumerate}

\noindent \textbf{Proof of (\ref{OI}):} Fix $d,k\in \N$, and suppose Opt$(d,k-1)$ and H-Opt$(d,k)$ hold. Suppose $P\subseteq \R^d$ determines at most $k$ distinct $\ell^1$-distances, and has maximal size amongst sets with this property. By scaling we can assume that the largest $\ell^1$-distance determined by $P$ is $2k$, which by Conjecture \ref{axisp} we know occurs in an axis-parallel direction. After permuting coordinates and translating, we can assume that $-ke_d,ke_d\in P$. As noted in (\ref{ballitem}), this implies that $P$ is contained the closed $\ell^1$-ball of radius $k$ centered at the origin. To verify this, suppose $U=(x_1,\dots,x_d)\in P$ with $\norm{U}_1 >k$. If $x_d\geq0$, then $\norm{U-(-ke_d)}_1=\norm{U}_1+k>2k$, while if $x_d\leq 0$, the same holds for the $\ell^1$-distance between $U$ and $ke_d$, contradicting the fact that $2k$ is the largest $\ell^1$-distance determined by $P$.

\noindent As noted in (\ref{capt}), the $\ell^1$-distance $2k$ can be eliminated from $P$ by removing the points in $P\cap H$, where $H$ is the closed upper $\ell^1$-hemisphere of radius $k$ centered at the origin. This is because, within a closed $\ell^1$-ball of radius $k$, the $\ell^1$-distance $2k$ only occurs between pairs of points on opposite faces of the boundary. By H-Opt$(d,k)$, we know that $|P\cap H|\leq |\Lambda_d(k)\cap H|$, and further $|P\cap H|= |\Lambda_d(k)\cap H|$ if and only if $P\cap H=\Lambda_d(k)\cap H$, in which case the $\ell^1$-distances determined by $P$ are $2,4,\dots,2k$.

\noindent Because the $\ell^1$-distance $2k$ does not occur in $P\setminus H$, we have that $P\setminus H$ determines at most $k-1$ distinct $\ell^1$-distances. By Opt$(d,k-1)$, we know that $|P\setminus H|\leq |\Lambda_d(k-1)|$, and further $|P\setminus H|= |\Lambda_d(k-1)|$ if and only if $P\setminus H$ is $\ell^1$-similar to $\Lambda_d(k-1)$. In order for both $P\cap H$ and $P\setminus H$ to attain their maximum possible sizes, the $\ell^1$-distances determined by $P\setminus H$ must be $2,4,\dots,2k-2$. In this case, since an $\ell^1$-similar copy of $\Lambda_d(k-1)$ is uniquely determined by its ``south pole" and its largest distance, we know that $P\setminus H$ must be $\Lambda_d(k-1)$ shifted down by $1$. In other words, \begin{align*} P\setminus H&= \{(x_1,\dots,x_{d}-1)\in \Z^d: |x_1|+\cdots+|x_d|\leq k-1, \ x_1+\dots+x_d\equiv k-1 \ (\text{mod }2)\} \\ &=\{(x_1,\dots,x_{d})\in \Z^d: |x_1|+\cdots+|x_d+1|\leq k-1, \ x_1+\dots+x_d\equiv k \ (\text{mod }2)\}.
\end{align*} The latter description ensures that $P\setminus H \subseteq \Lambda_d(k) \setminus H$, and conversely, if $U=(x_1,\dots,x_d) \in \Lambda_d(k)\setminus H$, then either $x_d<0$ or $\norm{U}_1\leq k-2$. In either case $|x_1|+\cdots+|x_d+1|\leq k-1$, and hence $U\in P\setminus H$. Bringing everything together, we have that if $P\cap H$ and $P\setminus H$ both attain their maximum possible size, then $P\cap H= \Lambda_d(k)\cap H$ and $P\setminus H=\Lambda_d(k)\setminus H$, hence $P=\Lambda_d(k)$, so Opt$(d,k)$ holds.

\noindent \textbf{Proof of (\ref{SH}):} \noindent Fix $d\geq 2$, suppose S-Opt$(d-1,k)$ holds for all $k\in \N$, and fix $k\in \N$. Suppose $P\subseteq \R^d$ is contained in the $\ell^1$-sphere $S$ of radius $k$ centered at the origin, and that $P$ has maximal size amongst all such sets determining at most $k$ distinct $\ell^1$-distances. To establish S-Opt$(d,k)$, we must show that $P=\Lambda_d(k)\setminus \Lambda_d(k-2)$. Thanks to our inductive hypothesis, we can assume $P$ determines exactly $k$ distinct $\ell^1$-distances, not fewer, and we denote those $\ell^1$-distances by $\lambda_1<\cdots<\lambda_k$. By the $\ell^1$-sphere component of Conjecture \ref{axisp}, we know that $\lambda_k$ occurs in an axis-parallel direction. By permuting coordinates, we assume that $\lambda_k$ occurs in the last coordinate direction, in other words $(x_1,\dots,x_d),(x_1,\dots,x_d+\lambda_k)\in P$ for some $x_1,\dots,x_d\in \R$ with $|x|+\cdots+|x_d|=|x_1|+\cdots+|x_d+\lambda_k|=k$, which in particular forces $x_d=-\lambda_k/2$ and $|x_1|+\cdots+|x_{d-1}|=k-\lambda_k/2$. This transformation is allowable because our end goal, $\Lambda_d(k)\setminus \Lambda_d(k-2)$, is invariant under coordinate permutation.  

\noindent We argue (informally for the moment) that the only reasonable choice is $\lambda_k=2k$ and $x_1=\cdots=x_{d-1}=0$, meaning $ke_d,-ke_d\in P$. This is because, since $\lambda_k$ is the largest $\ell^1$-distance, $P$ is contained in the intersection of the closed $\ell^1$-balls of radius $\lambda_k$ centered at $(x_1,\dots,x_{d-1},-\lambda_k/2)$ and $(x_1,\dots,x_{d-1},\lambda_k/2)$, respectively, and this intersection is the $\ell^1$-ball of radius $\lambda_k/2$ centered at $(x_1,\dots,x_{d-1},0)$. However, $P$ is also contained in $S$, so if it is not the case that $\lambda_k=2k$, then $P$ would in fact be contained in the intersection of an $\ell^1$-sphere with a closed $\ell^1$-ball of a smaller radius, which is at most a closed $\ell^1$-hemisphere. The idea that a maximal subset of an $\ell^1$-sphere determining at most $k$ distinct $\ell^1$-distances could actually be contained in a closed $\ell^1$-hemisphere is intuitively suspect, and we return to this issue near the end of the proof. For now, we assume $-ke_d,ke_d\in P$.

\noindent We let $H$ denote the closed upper $\ell^1$-hemisphere of $S$, and we establish H-Opt$(d,k)$ along the way. As discussed in (\ref{hyper}), all the points of $(P\cap H)\setminus ke_d$  have $x_d$-coordinates in the list $c_1> c_2 > \cdots > c_k$, where $c_i=k-\lambda_i/2$. For each $c_i$, the points of $H$ with $x_d$-coordinate equal to $c_i$ take the form $(x_1,\dots,x_{d-1},c_i)$ where $|x_1|+\cdots+|x_{d-1}| = k-c_i$, and we refer to the set of such points as $S_i$. With regard to $\ell^1$-distances, $S_i$ is equivalent to an $\ell^1$-sphere in $\R^{d-1}$, centered at the origin with radius $k-c_i$. All $\ell^1$-distances determined by $P\cap S_i$ are at most $2(k-c_i)=\lambda_i$, so $P\cap S_i$ determines at most $i$ distinct $\ell^1$-distances. By our inductive hypothesis, $|P\cap S_i| \leq |\Lambda_{d-1}(i)|-|\Lambda_{d-1}(i-2)|$, with equality holding if and only if the projection $P\cap S_i$ onto the first $d-1$ coordinates is $\ell^1$-similar to $\Lambda_{d-1}(i)\setminus \Lambda_{d-1}(i-2)$, so in particular  $\lambda_1,\dots,\lambda_i$ form an arithmetic progression.  Since $k-z_i=\lambda_i/2$ and $\lambda_k=2k$, this equality holds for all $1\leq i \leq k$ if and only if $P\cap S_i=\{(x,k-i): x\in \Lambda_{d-1}(i)\setminus \Lambda_{d-1}(i-2)\}$ for all $1\leq i \leq k$. Here we note that if it were not the case that $ke_d,-ke_d\in P$ as previously assumed, then $P\cap S_i$ would be, at most, equivalent to an $\ell^1$-hemisphere in $\R^{d-1}$, in which case our inductive hypothesis would prohibit it from having $|\Lambda_{d-1}(i)|-|\Lambda_{d-1}(i-2)|$ elements.

\noindent In summary, \begin{equation}\label{hemiub}|P\cap H| \leq \sum_{i=0}^{k} |\Lambda_{d-1}(i)|-|\Lambda_{d-1}(i-2)|,\end{equation} taking $\Lambda_{d-1}(-1)$ and $\Lambda_{d-1}(-2)$ to be empty, and equality holds if and only if $$P\cap H = \bigcup_{i=0}^k\left\{(x,k-i): x\in \Lambda_{d-1}(i)\setminus \Lambda_{d-1}(i-2)\right\}= \Lambda_d(k)\cap H,$$ which establishes H-Opt$(d,k)$. 

\noindent Further, $P\cap H$ can contain at most $\sum_{i=0}^{k-1} |\Lambda_{d-1}(i)|-|\Lambda_{d-1}(i-2)|$ points with $x_d>0$. Letting $H'$ denote the closed lower $\ell^1$-hemisphere of $S$, we employ the identical reasoning as above to yield the same upper bound (\ref{hemiub}) on $|P\cap H'|$, with equality holding if and only if $$P\cap H' = \bigcup_{i=0}^{k}\{(x,i-k): x\in \Lambda_{d-1}(i)\setminus \Lambda_{d-1}(i-2)\}=\Lambda_d(k)\cap H.$$ Further, $P\cap H'$ can contain at most $\sum_{i=0}^{k-1} |\Lambda_{d-1}(i)|-|\Lambda_{d-1}(i-2)|$ points with $x_d<0$. Putting all this together, we have \begin{align*}|P|&=|P\cap (H\setminus H')|+|P\cap (H' \setminus H)|+|P\cap H\cap H'| \\&\leq 2\left(\sum_{i=0}^{k-1} |\Lambda_{d-1}(i)|-|\Lambda_{d-1}(i-2)|\right)+\left(|\Lambda_{d-1}(k)|-|\Lambda_{d-1}(k-2)|\right), \end{align*} and equality holds if and only if $$ P= \bigcup_{i=-k}^k\left\{(x,i): x\in \Lambda_{d-1}(k-|i|)\setminus \Lambda_{d-1}(k-|i|-2)\right\} =\Lambda_d(k)\setminus \Lambda_d(k-2).$$ Therefore, S-Opt$(d,k)$ holds, and the induction on dimension is complete. \end{proof}

\noindent \textbf{Remark.} If one is specifically interested in dimension $d=3$, then, because we have fully resolved the problem in dimension $d=2$, no inductive hypothesis is needed for dimension, just for the number of $\ell^1$-distances. In other words, if Opt$(3,k-1)$ holds, then $\Lambda_3(k)$ is the unique set, up to $\ell^1$-similarity, of maximal size amongst axis-parallel subsets of $\R^3$ determining at most $k$ distinct $\ell^1$-distances. In particular, because Theorem \ref{singdist} tells us that O$(3,1)$ holds, we know that $\Lambda_3(2)$, which contains $19$ points and is pictured in Figure \ref{fig:sub2}, is uniquely optimal amongst axis-parallel sets determining only two $\ell^1$-distances. However, we cannot make the analogous claim for $\Lambda_3(3)$, because we cannot exclude the possibility of a non-axis-parallel set determining two $\ell^1$-distances that contains more than $19$ points (or that contains exactly $19$ points but is not $\ell^1$-similar to $\Lambda_3(2)$). This possibility disables our bridge from two $\ell^1$-distances to three, and exemplifies the need in assuming Conjecture \ref{axisp} if we wish to glean additional information in dimension $d\geq 3$. 

\noindent \textbf{Acknowledgements:}  This research was initiated during the Summer 2019 Kinnaird Institute Research Experience at Millsaps College. All authors were supported during the summer by the Kinnaird Endowment, gifted to the Millsaps College Department of Mathematics. At the time of submission, all authors except Alex Rice were Millsaps College undergraduate students. The authors would like to thank Alex Iosevich for his helpful references, and Tomasz Tkocz for alerting us to previous research done in the $k=1$ case. Finally, the authors would like to thank the anonymous referee for their encouraging comments and helpful recommendations, particular regarding Sections 5 and 6.


\begin{thebibliography}{10} 

\bibitem{Band} {\sc H.-J. Bandelt, V. Chepoi, M. Laurent}, {\em Embedding into rectilinear spaces}, Discrete Comput. Geom 19 (1998), 595-604.

\bibitem{Erdos} {\sc P. Erd\H{o}s}, {\em On sets of distances of $n$ points}, Amer. Math. Monthly 53 (1946), 248-250.

\bibitem{EF} {\sc P. Erd\H{o}s, P. Fishburn}, {\em Maximum planar sets determining $k$ distances}, Discrete Math. 160 (1996), 115-125. 

\bibitem{ErdosBook} {\sc J. Garibaldi, A. Iosevich, S. Senger}, {\em The Erd\H{o}s Distance Problem}, Volume 56 of the Student Mathematical Library, American Math. Soc., 2011.

\bibitem{GuthKatz} {\sc L. Guth, N. Katz}, {\em On the Erd\H{o}s distinct distances problem in the plane}, Ann. of Math. 181 (2015), no. 1, 155-190.

\bibitem{Guy} {\sc R. K. Guy}, {\em An olla-podria of open problems, often oddly posed}, Amer. Math. Monthly 90 (1983), 196-199.

\bibitem{Koolen} {\sc J. Koolen, M. Laurent, A. Schrijver}, {\em Equilateral dimension of the rectilinear space}, Designs, Codes Cryptogr. 21 (2000), 149-164.

\bibitem{Shin} {\sc M. Shinohara}, {\em Uniqueness of maximum planar five-distance sets}, Discrete Math. 308 (2008), no. 14, 3048-3055.

\bibitem{Wei} {\sc X. Wei}, {\em A proof of Erd\H{o}s-Fishburn's conjecture for $g(6)=13$}, Electron. Journ. of Combin. 19 (2012), no. 4, Paper \#38. 


 



\end{thebibliography}
\end{document}